\documentclass[10pt]{amsart}

\usepackage{lmodern}
\usepackage[T1]{fontenc}

\usepackage{amsmath, amsthm, amscd, amssymb, xcolor, fancyhdr}
\definecolor{dblue}{rgb}{0,0,.6}
\usepackage[backref=page, colorlinks=true, linkcolor=dblue, citecolor=dblue, filecolor=dblue, menucolor=dblue, urlcolor=dblue]{hyperref}
\usepackage{tikz-cd}

\renewcommand*{\backrefalt}[4]{
	\ifcase #1 (Not cited)
	\or        (Cited on page~#2)
	\else      (Cited on pages~#2)
	\fi}

\newcommand{\bbZ}{\mathbb{Z}}
\newcommand{\bbQ}{\mathbb{Q}}
\newcommand{\bbR}{\mathbb{R}}
\newcommand{\bbC}{\mathbb{C}}

\newcommand{\bbP}{\mathbb{P}}
\renewcommand{\AA}{\mathcal{A}}

\newcommand{\DD}{\mathcal{D}}

\newcommand{\HH}{\mathcal{H}}
\newcommand{\LL}{\mathcal{L}}
\newcommand{\MM}{\mathcal{M}}

\newcommand{\XX}{\mathcal{X}}
\newcommand{\YY}{\mathcal{Y}}

\newcommand{\frgt}{\mathfrak{g}_{\mathrm{tot}}}

\newcommand{\frso}{\mathfrak{so}}

\renewcommand{\ge}{\geqslant}

\newcommand{\catMmot}{\mathsf{(Mot_{A})}}
\newcommand{\catMab}{\mathsf{(Mot_A^{ab})}}

\newcommand{\bfM}{\mathbf{M}}
\newcommand{\bfH}{\mathbf{H}}

\newcommand{\st}{\enskip |\enskip}
\newcommand{\sdot}{{\raisebox{0.16ex}{$\scriptscriptstyle\bullet$}}}

\newcommand{\emrp}{\mathrm{End}}

\newcommand{\ii}{i}

\newcommand{\Spin}{\mathrm{Spin}}
\newcommand{\lrarr}{\longrightarrow}
\newcommand{\hrarr}{\hookrightarrow}

\newtheorem{defn}{Definition}[section]
\newtheorem{prop}[defn]{Proposition}
\newtheorem{thm}[defn]{Theorem}

\newtheorem{cor}[defn]{Corollary}

{\theoremstyle{remark}
  \newtheorem{rem}[defn]{Remark}
  
}

\title{Cohomology and Andr\'e motives of hyperk\"ahler orbifolds}

\author{Andrey Soldatenkov}
\address{Steklov Mathematical Institute of Russian Academy of Sciences\\
8 Gubkina Street\\ Moscow 119991\\ Russia}
\email{soldatenkov@mi-ras.ru}
\thanks{This work was performed
at the Steklov International Mathematical Center and supported
by the Ministry of Science and Higher Education of the Russian Federation
(agreement no. 075-15-2019-1614). The author is supported in part by Young Russian Mathematics award.}

\begin{document}

\begin{abstract}
One of the main tools for the study of compact hyperk\"ahler manifolds
is the natural action of the Looijenga--Lunts--Verbitsky Lie algebra on the cohomology
of such manifolds.
This also applies to the
mildly singular holomorphic symplectic varieties---hyperk\"ahler orbifolds,
allowing us to generalize the results of \cite{S1} to this setting.
\end{abstract}

\maketitle

\section{Introduction}

It has recently been established that the theory of irreducible holomorphic symplectic
manifolds may be generalized in a reasonable way to the situation when the
symplectic varieties are singular. The basics of the theory have been developed
by Beauville, Kaledin, Namikawa and others (see e.g. \cite{Be}, \cite{Ka}, \cite{Na} and references therein),
and more recently a version of the global Torelli
theorem for singular symplectic varieties has been proven by Bakker and Lehn \cite{BL}.
More can be said when the the symplectic varieties have only quotient singularities.
In this case the varieties possess metric properties similar to the smooth case,
namely they admit hyperk\"ahler metrics, see \cite{Ca}. This fact leads to a number
of consequences that have not been fully covered in the literature yet.
The purpose of this paper is to develop some of these aspects of the theory of
hyperk\"ahler orbifolds.

Existence of hyperk\"ahler metrics on symplectic orbifolds implies that their
singular cohomology with rational coefficients admits a natural Lie algebra action,
similar to the smooth case. Moreover, purity of the Hodge structures allows
us to define Andr\'e motives of hyperk\"ahler orbifolds and to generalize
the results of \cite{S1} to the orbifold setting. The main result is Theorem \ref{thm_main}
which allows us to prove that many of the known examples of projective
irreducible hyperk\"ahler orbifolds with $b_2\ge 4$ have abelian Andr\'e motives.
The importance of this statement stems from its relation to the Hodge
and Mumford--Tate conjectures, see \cite{S1}, \cite{F1}, \cite{F2}, \cite{FFZ}
for a detailed discussion of these relations in the case of smooth hyperk\"ahler
varieties.

\section{Hyperk\"ahler orbifolds}

\subsection{Complex orbifolds}

For a complex analytic variety $X$ its singular locus
will be denoted by $X^{\mathrm{sg}}$ and its smooth locus
by $X^{\mathrm{sm}}$.
For the purposes of this paper the term orbifold will be
synonymous with variety that has only quotient singularities.
More precisely, we will use the following definition.

\begin{defn}
Let $X$ be a normal, connected, complex analytic variety of dimension $n$. We will call $X$
an orbifold if any point of $X$ has an open neighbourhood
isomorphic to $U/G$, where $G\subset GL_n(\bbC)$ is a finite group
and $U\subset \bbC^n$ is an open $G$-invariant neighbourhood of zero.
\end{defn}

\begin{rem}
If $G_0\subset G$ is the subgroup generated by all pseudoreflections
contained in $G$, then $G_0$ is normal in $G$ and $U/G_0$ is smooth.
Therefore we may assume in the above definition that no $g\in G$,
$g\neq 1$ fixes a hyperplane. Under this assumption for any point
$x\in X$ the group $G$ is uniquely recovered as the local
fundamental group of $X^{\mathrm{sm}}$ around $x$.
\end{rem}

For a local chart $U/G \to X$ of an orbifold $X$ denote by $U^\circ$
the preimage of $X^{\mathrm{sm}}$ in $U$.
A differential form on $X$ is a differential form on $X^{\mathrm{sm}}$
whose pull-back to any local chart $U^\circ$ extends to a differential form on $U$.
A two-form is called K\"ahler if this extension is a K\"ahler form on $U$
for any chart. It has been shown by Satake \cite{Sa} that a number of
results analogous to the smooth case hold for orbifolds. In particular,
de Rham theorem, Poincar\'e duality with rational coefficients, Hodge
and Lefschetz decompositions of cohomology are among such results.

\subsection{Beauville-Bogomolov decomposition and hyperk\"ahler orbifolds}

By the fundamental group of an orbifold $X$ we will mean the group $\pi_1(X^{\mathrm{sm}})$,
and we will call $X$ simply connected if this group is trivial. A morphism of
complex analytic spaces $f\colon Y\to X$, where $X$, $Y$ are orbifolds, is called
an orbifold covering if $f$ is a topological covering over $X^{\mathrm{sm}}$.

We recall the orbifold version of the Beauville--Bogomolov decomposition theorem.
A simply connected $n$-dimensional orbifold $X$ will be called Calabi--Yau if it admits a K\"ahler metric with holonomy
equal to $SU(n)$ and hyperk\"ahler if it admits a K\"ahler metric with holonomy equal to $Sp(n/2)$.

\begin{thm}[{\cite[Theorem 6.4]{Ca}}]
Let $X$ be a compact K\"ahler orbifold with $c_1(X) = 0$. Then there exists a finite
orbifold covering $\pi\colon \widetilde{X}\to X$ with
$$\widetilde{X}\simeq T\times \prod_i Y_i\times \prod_j Z_j,$$
where $T$ is a compact complex torus, $Y_i$ are Calabi--Yau orbifolds and $Z_j$ are hyperk\"ahler orbifolds.
\end{thm}

Assume that $X$ is a hyperk\"ahler orbifold with hyperk\"ahler metric $g$.
Let $I$ denote its complex structure. Since
the standard representation of $Sp(n/2)$ is quaternionionic,
the holonomy principle implies that $X$ admits two more complex structures
$J$ and $K$ with $IJ=-JI =K$. Integrability of these structures follows from
the fact that they are preserved by the Levi--Civita connection which is torsion-free.
We denote the corresponding K\"ahler forms by $\omega_I$, $\omega_J$ and $\omega_K$.
The form $\sigma_I = \omega_J+\ii\omega_K$ is holomorphic symplectic with respect to $I$.

A typical example of a hyperk\"ahler orbifold is a singular K3 surface obtained by
contracting $(-2)$-curves on a smooth K3 surface. Such orbifold admits a small deformation
that is again smooth. Hence it is more natural to think of such an example as
of a degeneration of a smooth hyperk\"ahler manifold, and we would like to focus
our attention on the case when the singularities can not be removed by deformation.
Therefore we introduce the following more restrictive notion.

\begin{defn}
A simply connected hyperk\"ahler orbifold $X$ satisfying the condition
$\mathrm{codim}_X X^{\mathrm{sg}}\ge 4$ will be called irreducible.
\end{defn}

The singularities of any compact complex orbifold are $\bbQ$-factorial (see \cite[Proposition 2.15]{BL}
for the discussion of $\bbQ$-factoriality in the complex analytic setting).
The singularities of any hyperk\"ahler orbifold are rational \cite{Be}.
For an irreducible hyperk\"ahler orbifold the singularities are moreover
terminal \cite[Theorem 3.4]{BL}.

\subsection{Deformations and moduli of irreducible hyperk\"ahler orbifolds}\label{sec_moduli}

Assume that $X$ is an irreducible hyperk\"ahler orbifold. A deformation of $X$
is a proper flat morphism of complex analytic spaces $\pi\colon \XX\to B$
with connected base $B$ and such that some fibre of $\pi$ is isomorphic to $X$.
We will denote the fibre of $\pi$ over a point $t\in B$ by $\XX_t$.
The condition on the codimension of $X^{\mathrm{sg}}$ implies (see \cite[Lemma 3.3]{Fu})
that any deformation of $X$ is locally trivial in the following sense.

\begin{defn} A morphism $\pi\colon \XX\to B$ of complex analytic spaces
is called locally trivial if for any point $x\in \XX$ some neighbourhood of $x$ in $\XX$ is biholomorphic
to $V_1\times V_2$ for some open neighbourhood $V_1$ of $x$ in the fibre $\XX_{\pi(x)}$
and some open neighbourhood $V_2$ of $\pi(x)$ in $B$, and $\pi|_{V_1\times V_2}$
is the projection to the second factor.
\end{defn}

Fix a basepoint $t_0$ such that $\XX_{t_0}\simeq X$.
It is known that locally trivial deformations of complex spaces are
trivial in the real analytic category, see \cite[Proposition 5.1]{AV}.
This means, in particular, that all fibres $\XX_t$ are isomorphic as
real analytic varieties and hence we have the monodromy action of $\pi_1(B,t_0)$
on $H^\sdot(X,\bbQ)$.

The moduli theory of hyperk\"ahler manifolds has been generalized to
the orbifold setting in \cite{Me}. As a consequence, it has been shown that
$H^2(X,\bbQ)$ carries a canonical symmetric bilinear form $q$ called
the BBF form, \cite[section 3.4]{Me}. This form has the following property:
there exists a constant $C_X\in \bbQ$ such that for any $a\in H^2(X,\bbQ)$
we have 
\begin{equation}\label{eqn_BBF}
q(a)^n = C_X \int_X a^{2n},
\end{equation}
where we use the intersection product on
the right hand side. The form $q$ is normalized so that it is integral
and primitive on the image of $H^2(X,\bbZ)$ in $H^2(X,\bbQ)$.
The signature of $q$ is $(3,b_2(X)-3)$.

We fix the lattice $\Lambda = H^2(X,\bbZ)/\mathrm{(torsion)}$ with the BBF form
and define the moduli space $\mathfrak{M}$ of marked hyperk\"ahler orbifolds
$(Y,\varphi)$ deformation equivalent to $X$, where
$\varphi\colon H^2(Y,\bbZ)/(\mathrm{torsion})\stackrel{\sim}{\to}\Lambda$
is a lattice isomorphism called marking, see \cite[section 3.5]{Me}.
For a symplectic form $\sigma\in H^0(Y,\Omega_Y^2)$ we have
$q(\sigma) = 0$ and $q(\sigma,\bar{\sigma}) > 0$.
Let $\DD\subset \bbP(\Lambda\otimes\bbC)$ be the period domain:
$$\DD = \{x\in\bbP(\Lambda\otimes\bbC) \st q(x) = 0,\, q(x,\bar{x}) > 0\}.$$
Then the period of $(Y,\varphi)$ is the point $\rho(Y,\varphi) = [\varphi(\sigma)]\in \DD$.

The moduli space $\mathfrak{M}$ is the set of isomorphism classes of
marked hyperk\"ahler orbifolds deformation equivalent to $X$ and
$\rho$ defines the period map $\mathfrak{M} \to \DD$.
The construction and the properties of $\mathfrak{M}$ in the orbifold
setting are completely analogous to the smooth case. In particular,
$\mathfrak{M}$ is a non-Hausdorff complex manifold of dimension $h^{1,1}(X)$
and $\rho$ is a local isomorphism. 

Recall the definition of the Hausdorff reduction of $\mathfrak{M}$.
We call two points $x,y\in \mathfrak{M}$ inseparable if for any
open neighbourhoods $U_x$ of $x$ and $U_y$ of $y$ we have
$U_x\cap U_y \neq \emptyset$. This turns out to be an equivalence
relation, and we define $\overline{\mathfrak{M}}$ to be
the set of equivalence classes. Then $\overline{\mathfrak{M}}$ is
a Hausdorff complex manifold, and clearly $\rho$ factors through $\overline{\mathfrak{M}}$.
For a fixed connected component $\MM$ of $\mathfrak{M}$ denote by
$\overline{\MM}$ its Hausdorff reduction.

\begin{thm}[{\cite[Theorem 5.9]{Me}}]\label{thm_torelli}
The period map $\rho$ defines an isomorphism between $\overline{\MM}$ and $\DD$.
\end{thm}

Below we will need to consider special subvarieties in the period domain
of the following form. Fix an element $h\in \Lambda$ such that $q(h)>0$.
Let 
$$\DD_h = \{x\in\bbP(\Lambda\otimes\bbC) \st q(x) = q(x,h) = 0,\, q(x,\bar{x}) > 0\} \subset \DD.$$
The points of $\DD_h$ are periods of orbifolds that admit a line bundle
with first Chern class $h$. All these orbifolds are projective by \cite[Theorem 1.2]{Me}.
Denote by $\mathfrak{M}_h$ and $\MM_h$ the preimages of $\DD_h$ in $\mathfrak{M}$
and $\MM$ respectively.

\subsection{The Looijenga--Lunts--Verbitsky Lie algebra action on cohomology}

If $X$ is an irreducible hyperk\"ahler orbifold with K\"ahler forms $\omega_I$, $\omega_J$
and $\omega_K$, let us denote the corresponding Lefschetz operators
by $L_I$, $L_J$ and $L_K$. Let $\Lambda_I$, $\Lambda_J$ and $\Lambda_K$ be the dual Lefschetz operators.
A pointwise computation identical to the case when $X$ is smooth (see e.g. \cite{S2})
shows that the dual Lefschetz
operators commute with each other and we have $[\Lambda_J, L_K] = W_I$
(and similarly for the cyclic permutations of $I$, $J$, $K$),
where $W_I$ is the operator that acts on the differential
forms of $I$-type $(p,q)$ as multiplication by $\ii (p-q)$.
Verbitsky \cite{V1} deduces from this that the Lie subalgebra of $\emrp(\Lambda^\sdot T^*\!X)$
generated by the above operators is isomorphic
to $\frso(4,1)$.

The Lefschetz operators and their duals act on the cohomology of $X$,
so we have an embedding of Lie algebras $\frso(4,1)\hrarr \emrp(H^\sdot (X,\bbC))$.
Let $\frgt(X)$ be the subalgebra of $\emrp(H^\sdot (X,\bbC))$ generated
by all such embeddings for all hyperk\"ahler metrics on $X$.
This is the ``total Lie algebra'' of Looijenga--Lunts \cite{LL}
and its structure was determined in \cite{V2}. 

We recall the description of $\frgt(X)$. Let $V = H^2(X,\bbQ)$
and let $\tilde{V} = \langle e_0\rangle \oplus V\oplus \langle e_4\rangle$ be graded with $e_k$ of degree $k$
and $V$ in degree 2. Define $\tilde{q} \in S^2\tilde{V}^*$ such that $\tilde{q}|_V = q$,
$\langle e_0, e_4\rangle$ orthogonal to $V$ with $q(e_0)=q(e_4)=0$ and $q(e_0,e_4) = 1$.

\begin{prop}[\cite{V2}, see also {\cite[Proposition 2.9]{S2}}]\label{prop_LLV}
The graded Lie algebra $\frgt(X)$ is isomorphic to $\frso(\tilde{V},\tilde{q})$.
The subalgebra $\frso(V,q)\subset\frgt(X)$ acts on $H^\sdot(X,\bbQ)$ by derivations.
\end{prop}

\subsection{Monodromy action on cohomology}

For an irreducible hyperk\"ahler orbifold $X$ denote by $\mathrm{Aut}^+(H^\sdot(X,\bbQ))$
the group of graded algebra automorphisms that act trivially in the top degree $H^{4n}(X,\bbQ)$.
It follows from
Proposition \ref{prop_LLV} that there exists a representation
\begin{equation}
\lambda\colon\Spin(V,q)\to \mathrm{Aut}^+(H^\sdot(X,\bbQ)).
\end{equation}

Let $\pi\colon \XX\to B$ be a deformation of $X$ with $\XX_{t_0} \simeq X$ for a basepoint $t_0\in B$.
For a covering $B'\to B$ we denote by $t_0'\in B'$ some preimage of $t_0$. The action of
$\pi_1(B,t_0)$ on the cohomology of $X$ clearly preserves the top degree cohomology classes.

\begin{prop}\label{prop_monodr}
There exists a finite covering $B'\to B$ such that the monodromy
representation $\pi_1(B',t_0')\to \mathrm{Aut}^+(H^\sdot(X,\bbQ))$
factors through $\lambda$. 
\end{prop}
\begin{proof}

Let $\mathrm{Aut}^\circ(H^\sdot(X,\bbQ))$ be the connected component of
the identity in the group $\mathrm{Aut}^+(H^\sdot(X,\bbQ))$.
After passing to a finite covering of $B$ we may assume that the monodromy
representation factors through $\mathrm{Aut}^\circ(H^\sdot(X,\bbQ))$.
Consider the action of $\mathrm{Aut}^\circ(H^\sdot(X,\bbQ))$ on $H^2(X,\bbQ)$.
It is clear from (\ref{eqn_BBF}) that this action preserves the BBF form $q$,
hence a homomorphism $\alpha\colon\mathrm{Aut}^\circ(H^\sdot(X,\bbQ))\to SO(V,q)$.

Let $\Gamma\subset \mathrm{Aut}^\circ(H^\sdot(X,\bbQ))(\bbQ)$ be a torsion-free arithmetic subgroup,
where the integral structure is determined by the integral cohomology of $X$.
The action of $\ker(\alpha)$ on $H^\sdot(X,\bbQ)$ preserves the Hodge structure
and the Hodge--Riemann bilinear forms. Therefore it preserves a positive definite
scalar product on $H^\sdot(X,\bbQ)$, and the real Lie group $\ker(\alpha)(\bbR)$
is compact. Since $\Gamma$ is torsion-free, it is mapped injectively into $SO(V,q)$
by $\alpha$. Passing to a finite index subgroup of $\Gamma$ we may assume that
$\alpha(\Gamma)$ is contained in the image of the morphism $\Spin(V,q)(\bbQ)\to SO(V,q)(\bbQ)$.
As a conclusion, we may assume that $\Gamma$ is contained in the image of $\lambda$.

Since the monodromy action preserves the integral cohomology of $X$,
the image of $\rho\colon\pi_1(B,t_0)\to \mathrm{Aut}^\circ(H^\sdot(X,\bbQ))$ is contained
in an arithmetic subgroup. Then $\rho^{-1}(\Gamma)$ is a finite index subgroup
of $\pi_1(B,t_0)$, and we define $B'$ to be the corresponding finite covering of $B$.
\end{proof}

\section{Andr\'e motives of varieties with quotient singularities}

\subsection{Motivated cycles}
Let $X$ be a non-singular complex projective variety, $n=\dim_\bbC(X)$.
If $\LL$ is an ample line bundle on $X$ with the first Chern class $h\in H^{2}(X,\bbZ)$,
the Lefschetz operator $L_h\in \emrp(H^\sdot(X,\bbC))$ induces isomorphisms
$$L_h^k\colon H^{n-k}(X,\bbC)\stackrel{\sim}{\to} H^{n+k}(X,\bbC)$$
for every $k=0,\ldots,n$.
Denote by $H^{k}_{\mathrm{pr}}(X,\bbC)\subset H^{k}(X,\bbC)$ the kernel of $L_h^{n-k+1}$ for
$k = 0,\ldots,n$. Denote by $*_h\in \emrp(H^\sdot(X,\bbC))$ the Lefschetz involution:
for $x\in H^{k}_{\mathrm{pr}}(X,\bbC)$ and $i=0,\ldots,n-k$ we have $*_h(L_h^ix) = L_h^{n-k-i}x$.

Consider two non-singular complex projective varieties $X$, $Y$ with two ample line bundles $\LL_1\in \mathrm{Pic}(X)$,
$\LL_2\in\mathrm{Pic}(Y)$. Let $h = c_1(p_X^*\LL_1\otimes p_Y^*\LL_2)\in H^2(X\times Y,\bbZ)$.
For arbitrary classes of algebraic cycles $\alpha,\beta\in H^\sdot_{\mathrm{alg}}(X\times Y,\bbQ)$, consider the class
\begin{equation}\label{eqn_clmot}
p_{X*}(\alpha\cup *_h\beta),
\end{equation}
where $p_X\colon X\times Y\to X$ is the projection. Let $H^\sdot_{M}(X,\bbQ)$
be the subspace of $H^\sdot(X,\bbQ)$ spanned by the classes (\ref{eqn_clmot})
for all $Y$, $\LL_1$, $\LL_2$, $\alpha$, $\beta$ as above. Elements of $H^\sdot_M(X,\bbQ)$ will
be called motivated cycles.

\subsection{Andr\'e motives}
We recall the definition of Andr\'e motives from \cite{An}.
For two non-singular connected projective varieties $X$, $Y$
let $\mathrm{Cor}_M^k(X,Y) = H^{k+\dim(X)}_M(X\times Y,\bbQ)$ be
the space of motivated correspondences. Define the $\bbQ$-linear category $\catMmot$
whose objects are triples $(X,p,n)$, where $X$ is a variety as above,
$p\in \mathrm{Cor}_M^0(X,X)$, $p\circ p = p$, and $n\in \bbZ$.
Define the morphisms from $(X,p,n)$ to $(Y,q,m)$ to be the subspace
$$q\circ \mathrm{Cor}^{m-n}_M(X,Y)\circ p \subset \mathrm{Cor}^{m-n}_M(X,Y).$$
Cartesian product of varieties provides a tensor product in $\catMmot$,
and it is known \cite{An} that $\catMmot$ is a semi-simple graded neutral Tannakian category.
We will denote the Andr\'e motive of a variety $X$ as follows:
$\bfM(X)(n) = (X,[\Delta_X],n)\in \catMmot$, where $\Delta_X$ is the diagonal in $X\times X$.
The K\"unneth components $\delta_k$ of the diagonal $[\Delta_X]$ are motivated cycles
and we have the decomposition $\bfM(X)(n) = \oplus_k \bfH^k(X)(n)$,
where $\bfH^k(X)(n) = (X, \delta_k, n) \in \catMmot$.
The motives $\bfM(A)(n)$ for all abelian varieties $A$ generate
the full Tannakian subcategory of $\catMmot$ that we denote $\catMab$
and call the category of abelian motives.

Assume now that $X$ is a projective orbifold. Let $r\colon Y\to X$
be a resolution of singularities. Since Poincar\'e duality with rational
coefficients holds for $X$, the Hodge structure on the cohomology $H^k(X,\bbQ)$
is pure and the pull-back morphism $r^*\colon H^k(X,\bbQ)\to H^k(Y,\bbQ)$
is injective, see \cite[Theorem 8.2.4(iv) and Proposition 8.2.5]{D2}.
Therefore $H^k(X,\bbQ)$ is a sub Hodge structure in the cohomology
of a smooth projective variety. We define the Andr\'e motive of $X$
as a submotive of $\bfM(Y)$ in a similar way.

Namely, we consider a simplicial resolution $\mathcal{Y}_\sdot \to X$ with 
$\mathcal{Y}_i$ smooth projective varieties and $\mathcal{Y}_0 = Y$, see \cite[6.2.5]{D2}.
We have the exact sequence \cite[Proposition 8.2.5]{D2}
$$
0\lrarr H^{k}(X, \bbC) \stackrel{r^*}{\lrarr} H^{k}(Y,\bbC) \stackrel{\delta_0^*-\delta_1^*}{\lrarr} H^{k}(\mathcal{Y}_1,\bbC),
$$
where $\delta_0,\delta_1\colon \mathcal{Y}_1 \rightrightarrows \mathcal{Y}_0 = Y$ are the face maps.
Note that $\delta_i^*$ define morphisms of motives $\bfH^k(Y)\to \bfH^k(\mathcal{Y}_1)$.

\begin{defn}
In the above setting we define
$$\bfH^k(X) = \mathrm{ker}\left(\bfH^{k}(Y) \stackrel{\delta_0^*-\delta_1^*}{\lrarr} \bfH^{k}(\mathcal{Y}_1)\right).$$
\end{defn}

Strictly speaking, the motive $\bfH^k(X)$ in the above definition depends on the choice
of the resolution $Y$. But if $r'\colon Y'\to X$ is another resolution, we can find
a third resolution $Y''$ that dominates both $Y$ and $Y'$. Using this, we see that the
submotives of $\bfH^{k}(Y)$ and $\bfH^{k}(Y')$ in the above definition are canonically
isomorphic. Therefore $\bfH^k(X)$ is well defined up to a canonical isomorphism and
we may ignore dependence on the resolution.

\section{Deformation principle for hyperk\"ahler orbifolds}

\subsection{Constructing families of hyperk\"ahler orbifolds with a given central fibre}

We start this section from the key technical result of the paper. Assume that we
have a projective family $\pi\colon \XX\to B$ of hyperk\"ahler varieties over a
quasi-projective base $B$. A fibre $X_0$ of this family represents
a point $p\in\mathfrak{M}$ in the moduli space. Given any other point $p'\in \mathfrak{M}$
that is non-separated from $p$, we may replace the fibre $X_0$ by the corresponding
variety $X_0'$. This gives a new family $\XX'$, possibly non-projective. This
operation involves gluing two families over a punctured disc, and after such
gluing $\XX'$ may a priori not even be Moishezon (a typical situation
when this happens is the Shafarevich--Tate twist, see \cite{AR} and \cite{SV}).
However, in the case of families of hyperk\"ahler varieties this does not happen
and the total space $\XX'$ of the new family is Moishezon, i.e. an algebraic space in the sense of Artin.

\begin{prop}\label{prop_family}
Let $X$ be an irreducible hyperk\"ahler orbifold and $h\in H^{1,1}(X)\cap H^2(X,\bbZ)$
a cohomology class with $q(h) > 0$. Then there exists a flat locally trivial
family $\pi\colon \XX\to C$ over a connected quasi-projective curve $C$ and
a line bundle $\LL\in \mathrm{Pic}(\XX)$ such that $\LL|_{\XX_t}$ is ample for
a general $t\in C$ and for some $t_0\in C$ we have $\XX_{t_0}\simeq X$
and $c_1(\LL|_{\XX_{t_0}}) = h$. 
\end{prop}

\begin{proof}
We pick a marking $\varphi\colon H^2(X,\bbZ)/(\mathrm{torsion})\stackrel{\sim}{\to}\Lambda$
and identify $h$ with its image in $\Lambda$. Denote by $p\in\MM_h$
the point corresponding to $(X,\varphi)$, where $\MM_h$ is the corresponding
connected component of the moduli space $\mathfrak{M}_h$ (we use the
notation introduced in section \ref{sec_moduli}).

Note that by the projectivity criterion \cite[Theorem 6.9]{BL} the variety $X$
is projective. Moreover, for a very general point of $\MM_h$ the corresponding
variety $Y$ has Picard rank one, and the line bundle $L$ over $Y$ with first Chern class $h$
is ample. Embed the variety $Y$ into a projective space using some power of $L$.
Then $Y$ represents a point in the Hilbert scheme of that projective space.
Take the connected component of this Hilbert scheme containing the point $[Y]$
and denote by $\HH$ the open subset parametrizing locally trivial deformations
of $Y$. Let $\YY\to \HH$ be the universal family. After possibly passing to a
finite cover of $\HH$ we obtain the period map $\rho\colon \HH\to \DD_h/\Gamma$,
where $\Gamma$ is an arithmetic subgroup of the orthogonal group.

The period map $\rho$ is a dominant morphism of algebraic varieties.
It follows that we can find a curve $C^\circ\subset \HH$ with projective
closure $C$ such that $\rho(C)$ passes through $\overline{p}$,
where $\overline{p}$ is the image of $p$ in $\DD_h/\Gamma$. Let $t_0\in C$ be a point with $\rho(t_0) = \overline{p}$.
We let $\XX\to C$ be the restriction of the universal family over
the Hilbert scheme to $C$. We may assume that the curve $C$ is
general, so that for a general $t\in C$ the fibre $\XX_t$ has
Picard rank one, generated by an ample line bundle with first Chern class $h$.

The fibre $\XX_{t_0}$ may be not isomorphic to $X$. Let $\Delta\subset C$
be a small disc around $t_0$ and $\Delta^* = \Delta\setminus \{t_0\}$.
We use \cite[Lemma 6.2]{S1} to replace $\XX|_{\Delta}$ with another
family, isomorphic to $\XX|_{\Delta}$ over $\Delta^*$ and with central fibre isomorphic to $X$.
This completes the proof.
\end{proof}

\subsection{The main result} Proposition \ref{prop_family} allows us
to construct enough families of irreducible hyperk\"ahler orbifolds
to connect any two points in the moduli space. Then we can apply deformation
principle for motivated cycles to these families. Using the Kuga--Satake
construction we obtain the following result.

\begin{thm}\label{thm_main}
Assume that $X_1$ and $X_2$ are deformation equivalent projective irreducible hyperk\"ahler
orbifolds with $b_2\ge 4$. If $\bfM(X_1)\in\catMab$ then $\bfM(X_2)\in\catMab$.
\end{thm}
\begin{proof}
We choose two markings $\varphi_i\colon H^2(X_i,\bbZ)/(\mathrm{torsion})\stackrel{\sim}{\to}\Lambda$
such that $(X_1,\varphi_1)$ and $(X_2,\varphi_2)$ are in the same connected component $\MM$
of the moduli space. Using the construction \cite[section 6.2]{S1},
we reduce to the case when $(X_1,\varphi_1)$ and $(X_2,\varphi_2)$ lie in $\MM_h$
for some $h\in \Lambda$ with $q(h)>0$. Then we apply Proposition \ref{prop_family}
and obtain a locally trivial family $\pi\colon \XX\to C$ such that $X_1$ and $X_2$ are
two fibres, and there exists a line bundle $\LL\in\mathrm{Pic}(\XX)$ that is relatively ample over
a general point of $C$.

Using Proposition \ref{prop_monodr} we construct a family of Kuga--Satake abelian varieties
$\alpha\colon\AA\to C$ attached to $\pi\colon\XX\to C$, analogously to \cite[Lemma 6.3]{S1}.
Let $\rho\colon \widetilde{\XX}\to \XX$ be a simultaneous resolution of singularities
(see \cite[Lemma 4.9]{BL}), where $\tilde{\pi}\colon \widetilde{\XX}\to C$ is a smooth
family. Since the fibres of $\pi$ have quotient singularities, we have an embedding
of variations of Hodge structures:
$$
R^\sdot\pi_*\bbQ\hrarr R^\sdot\tilde{\pi}_*\bbQ
$$
We also have the Kuga--Satake embedding for the family $\XX$, see \cite{KSV} and \cite[Corollary 3.3]{S1}.
This embedding is given by a section of the local system $R^\sdot\psi_*\bbQ$,
where $\psi\colon \widetilde{\XX}\times_C \AA\to C$. Since the Andr\'e motive of
$X_1$ is abelian, the value of that section at one point is a motivated cycle.
By a version of the deformation principle proven in \cite[Proposition 5.1]{S1},
the value of the section at any point of $C$ is a motivated cycle, hence the
motive of $X_2$ is embedded into the motive of an abelian variety. This concludes the proof. 
\end{proof}

\subsection{Examples of irreducible hyperk\"ahler orbifolds}\label{sec_examples}

We apply Theorem \ref{thm_main} to two families of irreducible hyperk\"ahler orbifolds.

For the first example consider a $K3$ surface $S$ with a symplectic involution $f$.
Then $f$ has 8 fixed points. Consider the induced involution $f^{[2]}$ of the Hilbert square $S^{[2]}$.
The involution $f^{[2]}$ has 28 isolated fixed points and a fixed surface $S'\subset S^{[2]}$
that is isomorphic to the resolution of $S/f$.

Let $Y$ be the blow-up of $S^{[2]}$ in $S'$ and $X$ be the quotient of $Y$ by the action
of $f^{[2]}$. Note that by \cite{An} the motives of $S^{[2]}$, $S'$ and $Y$ are abelian,
hence $\bfM(X)\in\catMab$. It is clear that $X$ is a symplectic orbifold with only
isolated singularities. Since the involution $f^{[2]}$ fixes the exceptional divisor $E\subset Y$
of the blow-up $Y\to S^{[2]}$, the orbifold fundamental group of $X$ is trivial,
hence $X$ is irreducible. Note also that $b_2(X) = b_2(S^{[2]}/f^{[2]}) + 1$,
so a general deformation of $X$ is not obtained from a deformation of $S^{[2]}$.

A similar construction applies to the generalized Kummer variety $K^2\!A$ if we start
from an abelian surface $A$ with an involution $f$. We obtain an irreducible
orbifold $X'$ as a partial resolution of the quotient of $K^2A$ by an involution induced
by $f$. 

\begin{cor}
Any projective hyperk\"ahler orbifold deformation equivalent either to $X$
or to $X'$ described above has an abelian Andr\'e motive.
\end{cor}


\end{document}